\documentclass[11pt]{article}

\usepackage{amsmath ,amssymb,amsthm,epsfig,color,verbatim,url}
\usepackage{amsfonts,authblk}
\usepackage{mathrsfs}
\usepackage{amssymb}
\usepackage{url}
\usepackage{epsfig}
\usepackage{multirow}
\usepackage{array}
\usepackage{subcaption}

\usepackage{algorithm}
\usepackage{algorithmicx}
\usepackage{algpseudocode}
\usepackage{hyperref}

\usepackage{fullpage,array}

\usepackage{graphicx} 

\usepackage{mathtools}

\newtheorem{theorem}{Theorem}[section]
\newtheorem{corollary}[theorem]{Corollary}

\newtheorem{lemma}[theorem]{Lemma}

\newtheorem{theoremx}{Theorem}
\newtheorem{corollaryx}[theoremx]{Corollary}

\newtheorem*{theoremlift}{Theorem~\ref{thmx:multi-egs}}
\newtheorem*{theoremnucl}{Theorem~\ref{thmx:multi-egs-nucleus}}

\theoremstyle{definition}
\newtheorem{definition}[theorem]{Definition}

\newcommand{\E}{\mathbf{E}}

\newcommand{\N}{\mathcal{N}}

\newcommand{\Z}{\mathbb Z}
\newcommand{\GE}{G_\mathbf{E}}
\newcommand{\Aut}{\mathop{\rm Aut}\nolimits}
\newcommand{\Sym}{\mathop{\rm Sym}\nolimits}

\newcommand{\Stg}{\mathop{\rm St}\nolimits_G}
\newcommand{\be}{\mathrm e}

\DeclareMathOperator{\St}{St}

\usepackage{tikz}
\usetikzlibrary{automata, positioning, arrows}

\usepackage[linguistics]{forest}

\title{Liftability and Contracting Property of Multi-EGS Groups}
\author[1]{Arsalan Akram Malik}
\author[1]{Dmytro Savchuk}

\affil[1]{Department of Mathematics and Statistics\\
        University of South Florida\\
        4202 E Fowler Ave\\
        Tampa, FL 33620-5700\\
        \href{mailto:arsalanakram@usf.edu}{arsalanakram@usf.edu}, \href{mailto:savchuk@usf.edu}{savchuk@usf.edu}}

\begin{document}

\maketitle

\begin{abstract}
We provide sufficient conditions for the multi-EGS groups to be liftable and thus produce new examples of groups acting transitively on regular trees of finite degree stabilizing one of the ends, whose closures are scale groups as defined by Willis. Additionally, we explicitly compute the contracting nuclei of the groups in this class. We also specialize our results to the classes of multi-edge spinal group and EGS-groups.
\end{abstract}


\section{Introduction}
\label{sec:intro}
The classes of groups studied in this paper have their origins in the pioneering works of Grigorchuk~\cite{grigorchuk1980burnside}, Gupta and Sidki~\cite{gupta_s:burnside}. The \emph{GGS-groups} were introduced and studied from the early 1990s with the term initially introduced by Baumslag in~\cite{baumslag:topics93}.  These groups often share common unusual properties with Grigorchuk and Gupta-Sidki groups, such as being infinite finitely generated periodic groups, just infinite, or branch~\cite{vovkivsky:infinite_torsion_groups00,bar_gs:branch}. Several further extensions of the class of GGS-groups was introduced afterwards. First, the construction was extended by Pervova~\cite{pervova:profinite07} in 2007, who coined the term \emph{extended Gupta-Sidki groups} (\emph{EGS-groups}) and constructed the first examples of groups acting on rooted trees, whose closures in the corresponding tree automorphism groups do not coincide with their profinite completions. Each EGS-group contains a GGS-group as a subgroup and can be uniquely identified by it. Another generalization of GGS-groups, the class of \textit{multi-edge spinal groups}, was proposed by Alexoudas, Klopsch, and Thillaisundaram in 2016~\cite{alexoudas2016maximal}. Finally, the latter groups were generalized by Klopsch and Thillaisundaram~\cite{klopsch2018maximal} to \emph{generalized multi-edge spinal groups}, which were renamed to shorter \emph{multi-EGS groups} in~\cite{thillaisundaram2021profinite}. It was shown in~\cite{klopsch2018maximal} that under certain common conditions these groups are regular branch, just-infinite, and do not have a congruence subgroup property. Some recent papers have studied the sizes of congruence quotients in GGS-groups~\cite{fernandez2014ggs} and the exponents of these quotients in multi-EGS groups~\cite{maini2024multi}. We review the explicit connection between all these classes of groups in Section~\ref{sec:liftable}.

The class of \textit{liftable self-similar groups} was introduced by Grigorchuk and the second author in~\cite{grigorchuk2023liftable} in connection with groups of isometries of local fields. The groups in this class admit natural ascending HNN extensions with respect to their \textit{lifting endomorphisms} that act transitively on $(d+1)$-regular (unrooted) trees $\widetilde T_{d+1}$ (obtained as the union of a countable family of rooted $d$-ary trees $T_d$) preserving one of the ends and embed into the groups of isometries of local fields. Moreover, their closures in the automorphism groups of $\widetilde T_{d+1}$ preserving a fixed end under mild conditions are scale groups as defined by Willis in~\cite{willis2020scale}. The class of scale groups plays an important role in the theory of totally disconnected locally compact (TDLC) groups. It is shown in~\cite{grigorchuk2023liftable} that many well-known self-similar groups are liftable, including the Grigorchuk group, the Basilica group, and the lamplighter group. It was also shown that GGS-groups are also liftable as long as the defining vector is nonsymmetric in the sense that it has a zero component, such that the symmetric component is non-zero. 

The main purpose of this paper is to provide sufficient conditions for the multi-EGS groups to be liftable, thus providing new examples of liftable groups and at the same time suggesting new applications of multi-EGS groups. Each multi-EGS group acting on a $p$-ary rooted tree (for an odd prime $p$) is defined by a collection of $p$ subspaces (not all empty) of the vector space $(\Z/p\Z)^{p-1}$ that are defined by their bases $\E^{(l)}=\left\{\mathbf{e}_i^{(l)} = (e_{i,1}^{(l)}, \ldots,e_{i,p-1}^{(l)}),1\leq i\leq r_l\right\}$, $1\leq l\leq p$ where $r_l\geq 0$, $1\leq l\leq p$ represent the dimensions of the corresponding subspaces. The collection 
\begin{equation}
\label{eqn:datum}
\E = \left(\E^{(1)},\ldots, \E^{(p)}\right)
\end{equation} 
is called the \emph{datum} that defines a multi-EGS group $G_{\mathbf E}$. The explicit definition is given in Subsection~\ref{subsec:multi-egs}. The core of this paper is the following theorem.

\begin{theoremx}
\label{thmx:multi-egs}
For an odd prime $p$, let $\GE$ be a multi-EGS group defined by datum $\E$ given by~\eqref{eqn:datum} such that for some $m \in \{1,\ldots,p\}$, there exist $k \in \{1,\ldots,r_m\}$  and $j\in\{1,\ldots,p-1\}$ that satisfy:
    \begin{itemize}
        \item $e_{k,j}^{(m)} \neq 0$,
        \item $e_{i,p-j}^{(l)} = 0$ for all $i \in \{1,\ldots,r_l\}$ and $l \in \{1,\ldots,p\}$.
    \end{itemize}
    Then $\GE$ is liftable with a lifting $\sigma$ given by:
\begin{equation}
\label{eqn:sigma_multi-egs}
\sigma\colon\left\{
\begin{array}{llll}
a&\mapsto& \left(\bigl(b_k^{(m)}\bigr)^{a^{1-j}}\right)^f,&\text{where $f=\left(e_{k,j}^{(m)}\right)^{-1}\in\Z/p\Z$},\\
b_i^{(l)}&\mapsto& \bigl(b_i^{(l)}\bigr)^{a^{p-l+1}} \ ,& i \in \{1,\ldots,r_l\}, l \in \{1,\ldots,p\}.\\
\end{array}\right.
\end{equation}
\end{theoremx}

To better visualize the conditions of the above theorem, we can represent the collections $\E^{(l)}$ by $r_l\times (p-1)$ matrices $M^{(l)}$ whose rows are elements of $\E^{(l)}$. Then the conditions of the theorem can be restated as follows: there is a column index $j$ such that the $j$-th column in each $M^{(l)}$ is zero, and there is at least one nonzero entry in the column $(p-j)$ in at least one of $M^{(l)}$.

As an immediate corollary of Theorem~\ref{thmx:multi-egs} and Theorem~B from~\cite{grigorchuk2023liftable} we obtain embedding results for ascending HNN extensions of multi-EGS groups.

\begin{corollaryx}
\label{colx:hnn}
Let $G_{\mathbf{E}}$ be a multi-EGS group acting on a $p$-ary rooted tree $T_p$ and defined by the datum $\mathbf{E}$ given by~\eqref{eqn:datum} that satisfies conditions of Theorem~\ref{thmx:multi-egs}, with the lifting given by~\eqref{eqn:sigma_multi-egs}. Then
\begin{itemize}
\item[(i)] There is an embedding $\theta$ of the ascending HNN extension $\widetilde G_{\mathbf{E}}$ of $G_{\mathbf{E}}$ by $\sigma$
\begin{equation}
\label{eqn:hnn_pres}
\widetilde G_{\mathbf{E}}=G_{\mathbf{E}}*_\sigma=\langle G_{\mathbf{E}},t\mid \text{relations in}\ G_{\mathbf{E}}, tgt^{-1}=\sigma(g)\ \text{for all}\ g\in G_{\mathbf{E}}\rangle
\end{equation}
into the group of automorphisms of $\widetilde T_{p+1}$ preserving one of its ends $\omega$.
\item[(ii)] $G_{\mathbf{E}}$ is self-replicating, $\theta(\widetilde G)$ acts transitively on the set of vertices of $\widetilde T_{p+1}$, and the closure of $\theta(\widetilde G)$ in the group of automorphisms of $\widetilde T_{p+1}$ that preserve $\omega$ is a scale group.
\end{itemize}
\end{corollaryx}

In Section~\ref{sec:liftable} we also state Corollaries~\ref{cor:multi-edge} and~\ref{cor:egs} in which, for the sake of notation simplicity, we specialize a more general Theorem~\ref{thmx:multi-egs} to the classes of EGS-groups and multi-edge spinal groups.

We finish the paper by showing that the multi-EGS groups are contracting in Section~\ref{sec:contracting} and explicitly computing their contracting nuclei. This property was first introduced by Nekrashevych (see, for example,~\cite{nekrash:self-similar}) in connection with the holomorphic dynamics via the theory of iterated monodromy groups, but its origins trace back to the original construction of the Grigorchuk group~\cite{grigorch:burnside}. The class of contracting groups has also nice algebraic and algorithmic properties, and includes most interesting examples of groups acting on rooted trees, such as Grigorchuk group, Gupta-Sidki group, Basilica group, and many others. All groups in this class admit a polynomial-time algorithm solving the word problem with the polynomial degree depending on the size of the nucleus~\cite{savchuk:wp}. This algorithm, along with many computational routines for both contracting and self-similar groups, are implemented in two GAP packages~\cite{muntyan_s:automgrp,bartholdi:fr}.

Furthermore, contracting groups can be used as a platform for group-based post-quantum cryptography protocols (such as the Anshel-Anshel-Goldfeld protocol~\cite{anshel2003non}) while using \emph{nucleus portraits} to represent their elements uniquely and efficiently. The detailed description of this approach is laid out in~\cite{kahrobaei2024contracting}. Myasnikov and Ushakov have suggested in~\cite{myasnikov2008random} that a well-known class of heuristic attacks on group-based cryptosystems, which are broadly referred to as length-based attacks~\cite{hughes2003length}, generically works well against groups that have many free subgroups.  By the result of Nekrashevych~\cite{nekrashevych:contracting_no_free} contracting groups have no free nonabelian subgroups, so they constitute a natural class where such attacks may not be effective. In~\cite{kahrobaei2024contracting} Kahrobaei and the authors investigated the effectiveness of the length-based attack variants against the simultaneous conjugacy search problem in contracting groups and discovered that the efficiency of the attack generally decreases with increasing group nucleus size. In Section~\ref{sec:contracting} we prove the following proposition:

\begin{theoremx}
\label{thmx:multi-egs-nucleus}
    For an odd prime $p$, the multi-EGS group $\GE$ defined by the numerical datum $\E$ given by~\eqref{eqn:datum} is a contracting group with nucleus 
    \[\N_\mathbf{E} = \langle a \rangle \cup  \bigcup_{l=1}^p\langle b_i^{(l)}\colon 1 \leq i \leq r_l \rangle\] and $ |\N_\mathbf{E}| = p^{r_1} + \dots + p^{r_p}$.
\end{theoremx}
This result shows that it is easy to construct multi-EGS groups with large nuclei, and thus suggests that multi-EGS groups, used as platforms in cryptographic protocols, have potential to withstand against length-based-type attacks. The contracting property of the multi-EGS groups also follows from the fact that they are generated by bounded automata in the sense of Sidki~\cite{sidki:nofree}, which are proved to be contracting by Bondarenko and Nekrashevych in~\cite{bondarenko_n:post-critically_finite_self-similar_groups03}. But for some applications as shown above nucleus size plays an important role.

\noindent \textbf{Acknowledgements.} The authors greatly acknowledge the support of American Institute of Mathematics (AIM). Part of this work was done in June 2024 during the workshop ``Groups of dynamical origin'' funded by AIM.


\section{Preliminaries and notation}
\label{sec:prelim}

We recall some definitions and establish notation that will be used in rest of this paper. 

We use $X^*$ to denote the free monoid on the set $X=\{0,1,\ldots ,d-1\}$ and recall that $X^*$ can be endowed with the structure of a rooted $d$-ary tree $T(X)\cong T_d$ by defining $v$ to be adjacent to $vx$ for every $v\in X^*$ and $x\in X$. Finite words over $X$ naturally label vertices of $T(X)$. Given $n \in \mathbb{N}$, the set $X^n$ of words of length $n$ over $X$ corresponds to the $n$-th level of $T(X)$. Each element of the group $\Aut(T(X))$ of automorphisms of the tree preserves its root and adjacency of vertices.

\begin{definition}
\label{def:section}
    Given $g \in \Aut(T(X))$ and $x \in X$, for each $v \in X^*$ there is $v' \in X^*$ such that
    \[g(xv)=g(x)v'.\] 
    The map $g|_x\colon X^*\to X^*$ given by \[g|_x(v)=v'\] is an automorphism of $T(X)$ called the \emph{section} of $g$ at $x$. For a word $v=x_1x_2\ldots x_n\in X^*$ we define the section $g|_v$ of $g$ at $v$ as
    \[g|_v=g_{x_1}|_{x_2}|\ldots|_{x_n}.\]
\end{definition}

\begin{definition}
\label{def:self-similar}
A subgroup of $\Aut(T(X))$ is called \emph{self-similar} if it is closed under taking sections at the vertices of $T(X)$.
\end{definition}

The aforementioned definitions allow us to use the language of wreath recursions as for each self-similar group $G$ there is a natural embedding
\[G\hookrightarrow G \wr \Sym(X),\]
where $\wr$ denotes the permutational wreath product, and $\Sym(X)$ denotes the symmetric group on $X$. This embedding is given by
\begin{equation}
\label{eq:wreath}
G\ni g\mapsto (g|_0,g|_1,\ldots,g|_{d-1})\sigma_g\in G\wr \Sym(X),
\end{equation}
where $\sigma_g$ is the permutation of $X$ induced by the action of $g$ on first level of the tree. The decomposition at the first level of all generators a self-similar subgroup $G$ of $\Aut(T(X))$ under the embedding~\eqref{eq:wreath} is called the \emph{wreath recursion defining the group}. With a slight abuse of notation it is standard to identify an element of $G$ with its image in $G\wr\Sym(X)$ and write $g=(g|_0,g|_1,\ldots,g|_{d-1})\sigma_g$. 

Wreath recursion is particularly convenient for computing the sections of group elements. Indeed, the products and inverses of automorphisms can be found as follows. If $g=(g_0,g_1,\ldots,g_{d-1})\sigma_g$ and $h=(h_0,h_1,\ldots,h_{d-1})\sigma_h$ are two elements of $\Aut(T(X))$, then \begin{equation}
\label{eqn:gh}
gh=\left(g_0h_{\sigma_g(0)},g_1h_{\sigma_g(1)},\ldots,g_{d-1}h_{\sigma_g(d-1)}\right)\sigma_g\sigma_h    
\end{equation} 
and the wreath recursion of $g^{-1}$ is
\begin{equation}
\label{eqn:ginv}
g^{-1}=\left(g^{-1}_{\sigma_g^{-1}(0)},g^{-1}_{\sigma_g^{-1}(1)},\ldots,g^{-1}_{\sigma_g^{-1}(d-1)}\right)\sigma_g^{-1}.
\end{equation}

\vspace{5mm}


\subsection{Liftable groups}
\label{subsec:liftable}

We recall that for a group $G$ acting on $T(X)$ and $v\in X^*$ the stabilizer $\Stg(v)$ of vertex $v$ is the subgroup of $G$ consisting of all elements that fix $v$. For each $n\geq 1$ the stabilizer $\Stg(n)$ of level $n$ in $G$ is the subgroup of $G$ consisting of all elements that fix all vertices of level $n$. Stabilizers of levels are normal finite index subgroups of $G$ such that
\[\bigcap_{n\geq1}\Stg(X^n)=\{1\}.\]

For each $v\in X^*$ the map
\begin{equation}
\label{eqn:projection}
\begin{array}{llll}\pi_v\colon &\St_{\Aut(T(X))}(v)&\to&\Aut(T(X))\\
&g&\mapsto&g|_v
\end{array}
\end{equation}
defines a homomorphism that we will call \emph{projection}. For each subgroup $G$ of $\Aut(T(X))$ the homomorphism $\pi_v$ restricts to a homomorphism $\Stg(v)\to\Aut(T(X))$. Moreover, if $G$ is a self-similar group, then, since $G$ is closed under taking the sections, $\pi_v$ is a homomorphism from $\Stg(v)$ to $G$. In the case $X=\{0,1,\ldots,d-1\}$ the corresponding projections are denoted by $\pi_i$, $i=0,1,\ldots,d-1$.

\begin{definition}
\label{def:liftable}

A self-similar group $G$ acting on a tree $T_d=T(X)$ is called \emph{liftable} if there exist some $i\in X$ and a homomorphism $\sigma\colon G\to \Stg(i)$, called the \emph{lifting}, that is the right inverse of the projection map $\pi_i$ defined in equation~\eqref{eqn:projection}, that is, such that $\pi_i\circ\sigma$ is the identity on $G$.
\end{definition}

The idea of the lifting is related to the construction of finitely $L$-presented groups acting on rooted trees (e.g., Grigorchuk group), where the lifting endomorphism corresponds to the substitution used in the $L$-presentation of the group~\cite{grigorchuk2023liftable}.

From the definition of a liftable group it follows that the projection map $\pi_i\colon\Stg(i)\to G$ must be surjective, i.e. liftable groups are self-replicating as long as they act transitively on $X$. 


\section{Liftable extensions of Multi-EGS groups}
\label{sec:liftable}

\subsection{Multi-EGS groups}
\label{subsec:multi-egs}
In the rest of the paper for an odd prime $p$ we consider the alphabet $X=\{0,1,\ldots,p-1\}$. The multi-EGS groups are defined as follows.

\begin{definition}
\label{def:multi-egs}
    For an odd prime $p$ and $l \in \{1,\ldots,p\}$, let $\E^{(l)}$ denote a collection of $r_l\geq 0$ linearly independent vectors in $(\mathbb{Z}/p\mathbb{Z})^{(p-1)}$ with at least one $r_l>0$. We denote the vectors in this collection by $\mathbf{e}_i^{(l)} = (e_{i,1}^{(l)}, \ldots,e_{i,p-1}^{(l)})$ where $i \in \{1,\ldots,r_l\}$. Then a multi-EGS group $\GE$ is defined by the datum $\E = \left(\E^{(1)},\ldots, \E^{(p)}\right)$ to be the self-similar group acting on $X^*$ generated by the set $\{a\}\cup\{b_i^{(l)} \mid 1 \leq i \leq r_l, 1 \leq l \leq p\}$ with the following wreath recursions:
\[
    \begin{array}{lcl}
a&=&(1,1,\ldots,1,1)\varepsilon,\\
b_i^{(l)}&=&(a^{e_{i,p-l+1}^{(l)}},\ldots,a^{e_{i,p-1}^{(l)}}, b_i^{(l)},a^{e_{i,1}^{(l)}},\ldots,a^{e_{i,p-l}^{(l)}}), \ 1\leq i\leq r_l , \ 1 \leq l \leq p, 
\end{array}
\]
where $\varepsilon=(0,1,\ldots,p-1)$ denotes a long cycle in $\Sym(X)$.
\end{definition}


It follows immediately from the wreath recursion that all generators have order $p$. Also, linear independence of vectors in ${\mathbf{E}}^{(l)}$ implies that $B_l=\langle b_1^{(l)}, \ldots, b_{r_l}^{(l)} \rangle \cong (\Z/p\Z)^{r_l}$ is an elementary abelian $p$-group of order $p^{r_l}$. We will also use the convention that if $r_l=0$ then $\mathbf{E}^{(l)}$ is empty and $B_l$ is trivial.

We now prove Theorem~\ref{thmx:multi-egs} from the Introduction that we restate here for convenience.

\begin{theoremlift}
For an odd prime $p$, let $\GE$ be a multi-EGS group defined by datum $\E$ given by~\eqref{eqn:datum} such that for some $m \in \{1,\ldots,p\}$, there exist $k \in \{1,\ldots,r_m\}$  and $j\in\{1,\ldots,p-1\}$ that satisfy:
    \begin{itemize}
        \item $e_{k,j}^{(m)} \neq 0$,
        \item $e_{i,p-j}^{(l)} = 0$ for all $i \in \{1,\ldots,r_l\}$ and $l \in \{1,\ldots,p\}$.
    \end{itemize}
    Then $\GE$ is liftable with a lifting $\sigma$ given by:
\begin{equation*}
\sigma\colon\left\{
\begin{array}{llll}
a&\mapsto& \left(\bigl(b_k^{(m)}\bigr)^{a^{1-j}}\right)^f,&\text{where $f=\left(e_{k,j}^{(m)}\right)^{-1}\in\Z/p\Z$},\\
b_i^{(l)}&\mapsto& \bigl(b_i^{(l)}\bigr)^{a^{p-l+1}} \ ,& i \in \{1,\ldots,r_l\}, l \in \{1,\ldots,p\}.\\
\end{array}\right.
\end{equation*}
\end{theoremlift}

\begin{proof}
    In Proposition 3.9 of~\cite{klopsch2018maximal}, it is shown that for a multi-EGS group, the abelianization $\GE/\GE' \ \cong (\Z/p\Z)^{1+r_1+\dots+r_p}$. Since we have $\left(a^{e_{k,j}^{(m)}}\right)^f=a^{e_{k,j}^{(m)}f}=a$, using equalities~\eqref{eqn:gh} and~\eqref{eqn:ginv} we compute the following wreath recursions for $l \in \{1,\ldots,p\}$ and $ i \in \{1,\ldots,r_l\}$:

\[
\begin{array}{rclllllllll}

\bigl(b_i^{(l)}\bigr)^{a^{p-l+1}}&=&(b_i^{(l)},&a^{e_{i,1}^{(l)}},&\ldots,&a^{e_{i,p-j-1}^{(l)}},&1,&a^{e_{i,p-j+1}^{(l)}},&\ldots,&a^{e_{i,p-2}^{(l)}},&a^{e_{i,p-1}^{(l)}}),\\


\left(\bigl(b_k^{(m)}\bigr)^{a^{1-j}}\right)^f&=&(a,&a^{e_{k,j+1}^{(m)}f},&\ldots,&a^{e_{k,p-1}^{(m)}f},&({b_k^{(m)}})^{f},&a^{e_{k,1}^{(m)}f},&\ldots,&a^{e_{k,j-2}^{(m)}f},&a^{e_{k,j-1}^{(m)}f}),\\[1mm]

\text{positions}&&\phantom{(}0&1&\ldots&p-j-1&p-j&p-j+1&\ldots&p-2&p-1\\
\end{array}
\]
where the third row indicates positions of coordinates of the wreath recursions given in the first two rows.

We show now that the substitution $\sigma$ defined by~\eqref{eqn:sigma_multi-egs} extends to an injective endomorphism of $\GE$. Suppose we have a relator in $\GE$ represented by a word $w=w(a,b_i^{(l)})$ in the free group $F(a,b_i^{(l)})$ where $i \in \{1,\ldots,r_l\}$ and $l \in \{1,\ldots,p\}$. Then
\[\pi_i\circ\sigma(w)=\pi_i\left(w\left(\left(\bigl(b_k^{(m)}\bigr)^{a^{1-j}}\right)^f,\bigl(b_i^{(l)}\bigr)^{a^{p-l+1}}\right)\right)=\left\{
\begin{array}{ll}
w\left(a,b_i^{(l)}\right)=_{\GE}1,&\text{if position}\ n=0,\\[2mm]
w\left(({b_k^{(m)}})^{f},1\right),&\text{if position}\ n=p-j,\\[2mm]
w\left(a^{e^{(m)}_{j+n}f},a^{e^{(l)}_{i,n}}\right),&\text{otherwise.}
\end{array}
\right.\]
Since $\GE/{\GE}'\cong(\Z/p\Z)^{1+r_1+\dots+r_p}$ and $w=_{\GE}1$, we must have that the total exponents of $a$ and $b^{(l)}_i$'s in $w$ are all equal to 0 modulo $p$. This implies that $w(({b_k^{(m)}})^{f},1)$ and $w(a^{(e^{(m)}_{j+n})f},a^{e^{(l)}_{i,n}})$ are trivial in $\GE$ since $a$ and all $b^{(l)}_i$ have order $p$.
By construction, $\pi_0\circ\sigma = _{\GE}1$, so $\sigma$ is a lifting of $\GE$. Therefore, $\GE$ is liftable.

\end{proof}


\subsection{Multi-edge spinal groups }
\label{subsec:multi-edge} 

The class of multi-edge spinal groups defined in~\cite{alexoudas2016maximal} is a subclass of the class of multi-EGS groups and is a generalization of the class of GGS-groups. Given an odd prime $p$, multi-edge spinal groups act on the regular $p$-ary rooted tree $T_p = T(X)$.
\begin{definition}
\label{def:multi-edge}
For a fixed prime $p>2$, $1\leq r<p$ and an $r$-tuple
\begin{equation}
\label{eqn:E}    
{\mathbf{E}}^{(p)}=\left(\mathbf{e}_i = (e_{i,1},\ldots,e_{i,p-1}) \in (\Z/p\Z)^{p-1}\right)_{i=1,\ldots,r}
\end{equation}
of linearly independent vectors in $(\Z/p\Z)^{p-1}$, a \emph{multi-edge spinal group} $G_{\mathbf{E}^{(p)}} = \langle a,b_1,\ldots,b_r\rangle$ acts on a regular rooted tree $T_p=T(X)$ for $X=\Z/p\Z$, with the generators defined by the following wreath recursions:
\[
\begin{array}{lcl}
a&=&(1,1,\ldots,1,1)\varepsilon,\\
b_i&=&(a^{e_{i,1}},a^{e_{i,2}},\ldots,a^{e_{i,p-1}},b_i), \ 1\leq i\leq r, 
\end{array}
\]
where $\varepsilon=(0,1,\ldots,p-1)$ denotes a long cycle in $\Sym(X)$.
\end{definition}

The multi-edge spinal groups correspond to the multi-EGS groups defined by a datum $\E = \left(\E^{(1)},\ldots, \E^{(p)}\right)$ in which all collections $\E^{(i)}$ except the last one $\E^{(p)}$ are empty.

The following result follows from Theorem~\ref{thmx:multi-egs}.

\begin{corollary}

\label{cor:multi-edge}
    Let $p\geq3$ be a prime, $1\leq r<p$, and $G_{\mathbf{E}^{(p)}}$ be a multi-edge spinal group defined by an  $r$-tuple $\mathbf{E}^{(p)}$ of linearly independent vectors in $(\Z/p\Z)^r$ as denoted in~\eqref{eqn:E}, such that there exist $k \in \{1,\ldots,r\}$ and $j\in\{1,\ldots,p-1\}$ that satisfy: 
    \begin{itemize}
        \item $e_{k,j} \neq 0$,
        \item $e_{i,p-j} = 0$ for all $1\leq i\leq r$.
    \end{itemize}
    Then $\GE$ is liftable. A lifting is witnessed by the following $\sigma$:
    \begin{equation}
\label{eqn:sigma_multi-edge}
\sigma\colon\left\{
\begin{array}{llll}
a&\mapsto& (b_k^{a^{1-j}})^{e_{k,j}^{-1}},&\\[2mm]
b_i&\mapsto& b_i^a \ ,& i \in \{1,\ldots,r\}.\\

\end{array}\right.
\end{equation}
    
\end{corollary}

\subsection{EGS-groups}
\label{subsec:egs}

The class of EGS-groups defined in~\cite{pervova:profinite07} is another subclass of the class of multi-EGS groups. Given an odd prime $p$, EGS-groups act on a regular $p$-ary rooted tree $T_p = T(X)$. Each group in this class is defined by a non-zero vector in $(\Z/p\Z)^{p-1}$ and contains the corresponding GGS-group defined by the same vector.

\begin{definition}
\label{def:EGS}
For a fixed $d\geq 2$, an EGS-group $G_{\be}$ is defined by a non-zero vector $\be=(e_1,e_2,\ldots,e_{p-1})\in(\Z/p\Z)^{p-1}$ as the group acting on the regular rooted tree $T_p=T(X)$ for $X=\Z/d\Z$ with the following wreath recursion:
\[
\begin{array}{rcl}
a&=&(1,1,\ldots,1,1)\varepsilon,\\
b&=&(a^{e_1},a^{e_2},\ldots,a^{e_{p-1}},b),\\
c&=&(c,a^{e_1},a^{e_2},\ldots,a^{e_{p-1}}).
\end{array}
\]
The $\varepsilon=(0,1,\ldots,p-1)$ denotes a long cycle in $\Sym(X)$. We will denote the corresponding group as $G_{\be}=\langle a,b,c \rangle$.

\end{definition}

The EGS-groups correspond to multi-EGS groups defined by datum $\E = \left(\E^{(1)},\ldots, \E^{(p)}\right)$ in which all collections $\E^{(i)}$, $2\leq i\leq p-1$ are empty and both $\E^{(1)}$ and $\E^{(p)}$ consist of only one vector $\be$. The subgroup $\langle a,b\rangle$ of $G_{\be}$ is the GGS-group associate to $G_{\be}$. 

The following result follows from Theorem~\ref{thmx:multi-egs}.

\begin{corollary}
\label{cor:egs}
Let $p\geq 3$ be a prime and $G_{\be}$ be an EGS-group defined by a non-symmetric vector $\be=(e_1,e_2,\ldots,e_{p-1})\in (\Z/p\Z)^{p-1}$ with the property that there is $j\in\{1,2,\ldots,p-1\}$ such that $e_j\neq 0$ and $e_{p-j}=0$. Then $G_{\be}$ is a liftable group. A lifting is given by the following $\sigma$:
\begin{equation}
\label{eqn:sigma_EGS}
\sigma\colon\left\{
\begin{array}{l}
a\mapsto (b^{a^{1-j}})^{e_j^{-1}},\\
b\mapsto b^a,\\
c\mapsto c. \\
\end{array}\right.
\end{equation}
\end{corollary}


\section{Contracting property and nucleus of multi-EGS groups}
\label{sec:contracting}

Contracting groups can be defined as self-similar groups in which the length of the sections of any word decreases to a constant length when going down the tree. We will use here the following result from~\cite[Lemma~12.1.2]{nekrash:self-similar} that can also be used as a formal definition of contracting groups.

\begin{lemma}
\label{lem:nek}
    A self-similar group $G$ with generating set $S=S^{-1}, 1\in S$ is contracting if and only if there exists a finite set $\N\subset G$ and $k\geq 1$ such that for all $n_1n_2\in (S\cup\N)^2$ and $v\in X^k$: $(n_1n_2)|_{v} \in \N$.
\end{lemma}

Any set satisfying the conditions of Lemma~\ref{lem:nek} is called a \emph{quasinucleus} of $G$ and the minimal such set under inclusion is the \emph{nucleus} of $G$. The nucleus of a contracting group can also be characterized as the self-similar closure of the union all the cycles in the full automaton of the group. To find the nucleus of a group it is sufficient to find a quasinucleus satisfying conditions of Lemma~\ref{lem:nek} and then select elements from that set that are either sections of themselves or are sections of such elements.

We now prove Theorem~\ref{thmx:multi-egs-nucleus} from the Introduction that we also restate here for convenience.
\begin{theoremnucl}
    For an odd prime $p$, the multi-EGS group $\GE$ defined by the numerical datum $\E$ given by~\eqref{eqn:datum} is a contracting group with nucleus 
    \[\N_\mathbf{E} = \langle a \rangle \cup  \bigcup_{l=1}^p\langle b_i^{(l)}\colon 1 \leq i \leq r_l \rangle\] and $ |\N_\mathbf{E}| = p^{r_1} + \dots + p^{r_p}$.
\end{theoremnucl}

\begin{proof}
    We recall the wreath recursions for the generators of multi-EGS groups from Definition~\ref{def:multi-egs}:
    \[
    \begin{array}{lcl}
a&=&(1,1,\ldots,1,1)\varepsilon,\\
b_i^{(l)}&=&(a^{e_{i,p-l+1}^{(l)}},\ldots,a^{e_{i,p-1}^{(l)}}, b_i^{(l)},a^{e_{i,1}^{(l)}},\ldots,a^{e_{i,p-l}^{(l)}}), \ 1\leq i\leq r_l , \ 1 \leq l \leq p. 
\end{array}
\]

For each $1\leq l\leq p$ such that $r_l>0$, the group $B_l=\langle b_1^{(l)}, \ldots, b_{r_l}^{(l)} \rangle\subset\N_{\mathbf{E}}$ is isomorphic to $(\Z/p\Z)^{r_l}$ and has order $p^{r_l}$. Every element of this group has itself as a section at a vertex of the first level, so must be in the nucleus. Additionally, since there is $1\leq l\leq p$ such that $r_l>0$, there must be at least one nonzero component of $\mathrm{e}^{(l)}_1$, which yields a nontrivial element of $\langle a\rangle$ as a section of $b_1^{(l)}$. Since $b_1^{(l)}$ stabilizes all vertices of the first level, all elements of $\langle a\rangle$ will occur as sections of elements $(b_1^{(l)})^k$ for $0\leq k< p$ because $p$ is prime. Thus, all powers of $a$ are sections of elements of the nucleus and hence must be in the nucleus themselves. We conclude that the set $\N_\mathbf{E}$ is contained in the nucleus of $G_\mathbf{E}$.

To prove the converse inclusion, we show using Lemma~\ref{lem:nek} that $\N_\mathbf{E}$ is a quasinucleus of $G_\mathbf{E}$. Since $\N_\mathbf{E}$ is a symmetric generating set for $G_\mathbf{E}$, we only need to prove that the product of any two elements $n_1,n_2\in\N_\mathbf{E}$ is either in $\N_\mathbf{E}$ or contracts to $\N_\mathbf{E}$ at most at the second level of the tree (i.e., the sections of $n_1n_2$ at the vertices of the second level belong to $\N_\mathbf{E}$). There are four possibilities as follows:

\begin{description}
    \item[Case I.]  If both $n_1,n_2$ are powers of $a$, then $n_1n_2$ is also a power of $a$ and is in $\N_\mathbf{E}$. 
    \item[Case II.] If both $n_1,n_2$ are elements of $B_l$ for some $1\leq l\leq p$, then $n_1n_2$ is again in $B_l$, so is an element of $\N_\mathbf{E}$.
    \item[Case III.] If one of $n_1,n_2$ is a power $a$ and the other is in $B_l$ for some $1\leq l\leq p$, then the set of sections of $n_1n_2$ at the first level will coincide with the set of sections of $n_i$ that is not a power of $a$. Therefore, all sections of $n_1n_2$ at the vertices of the first level are in $\N_\mathbf{E}$.
    \item[Case IV.] If $n_1\in B_l$ and $n_2\in B_{l'}$ for some $l\neq l'$, then from the wreath recursions we see that each section of $n_1n_2$ will either be a power $a$, or an element of $B_l$ or $B_{l'}$, or a product of a power of $a$ and an element from either $B_l$ or $B_{l'}$. In the former cases we immediately obtain an element of $\N_\mathbf{E}$, while in the latter case we get to Case~III, which shows that the sections of $n_1n_2$ at the vertices of the second level are in $\N_\mathbf{E}$.
\end{description}
Thus, by Lemma~\ref{lem:nek} (with $k = 2$) the multi-EGS group $G_{\mathbf{E}}$ is contracting with the nucleus $\N_\mathbf{E}$. Since the groups $\langle a\rangle$ and $B_l, 1\leq l\leq p$ pairwise intersect at the identity of $G_{\mathbf{E}}$, the size of $\N_\mathbf{E}=\langle a\rangle\cup\bigcup_{l=1}^pB_l$ is 
\[|\N_\mathbf{E}|=1+(|\langle a\rangle|-1)+\sum_{l=1}^p(|B_l|-1)
=1+(p-1)+\sum_{l=1}^p(p^{r_l}-1)=\sum_{l=1}^pp^{r_l}.
\]
We finally note that the above argument and formula are alse valid when some of the $r_l$'s are 0 or, equivalently, when $B_l$'s are trivial subgroups of $G_{\mathbf{E}}$.
\end{proof}

Theorem~\ref{thmx:multi-egs-nucleus} specializes to the class of multi-edge spinal groups as follows. We use the notation from Definition~\ref{def:multi-edge}.

\begin{corollary}
\label{cor:multi-edge-nucleus}
    For a prime $p\geq3$ and $1\leq r<p$, the multi-edge spinal group $G_{\mathbf{E}^{(p)}}$ defined by an $r$-tuple $\mathbf{E}^{(p)}$ of linearly independent vectors in $(\Z/p\Z)^{p-1}$ is a contracting group with the nucleus $\N_{\mathbf{E}^{(p)}} = \langle a \rangle \cup \langle b_1,\ldots,b_r \rangle$ of size $p^{r}+p-1$.
\end{corollary}

In the case of EGS-groups Theorem~\ref{thmx:multi-egs-nucleus} takes even simpler form.
    
\begin{corollary}
\label{prop:egs-nucleus}
    For a prime $p>2$, an EGS-group $G_{\be}$, defined by a non-zero vector $\be=(e_1,e_2,\ldots,e_{p-1})\in(\Z/p\Z)^{p-1}$ is a contracting group with the nucleus $\mathcal N_{\be}=\{1,a^i,b^j,c^k\mid 1\leq i,j,k\leq p-1\}$ of size $3p-2$. 
\end{corollary}

The GAP code to generate some classes of groups studied in the paper is provided in~\cite{code:multi_EGS_github}.


\bibliographystyle{plain}

\end{document}